\documentclass{tac}

\usepackage{mathrsfs,amssymb,stmaryrd,bbm,mathtools, upgreek}
\usepackage{chngcntr}
\usepackage[shortlabels]{enumitem}
\makeatletter
\@addtoreset{equation}{section}
\makeatother
\numberwithin{equation}{subsection}

\newtheorem{theo}{Theorem}[section]
\newtheorem{lem}[theo]{Lemma}
\newtheorem{prop}[theo]{Proposition}
\newtheorem{coro}[theo]{Corollary}
\newtheorem{defi}[theo]{Definition}
\newtheoremrm{rem}[theo]{Remark}

\let\olddefi\defi
\renewcommand{\defi}{\olddefi\normalfont}
\let\oldrem\rem
\renewcommand{\rem}{\oldrem\upshape}

\def\mathrmdef#1{\expandafter\def\csname#1\endcsname{{\rm#1}}}
\def\mathsfdef#1{\expandafter\def\csname#1\endcsname{{\rm\sf#1}}}
\def\mathcaldef#1{\expandafter\def\csname#1\endcsname{{\mathcal#1}}}

\mathrmdef{op}\mathrmdef{co}\mathrmdef{coop}
\mathsfdef{CoAlg}   \mathsfdef{Cat}\mathsfdef{CAT}
\mathsfdef{Ord}
\mathsfdef{Set}

\newcommand\objj[1]{\mathrm{obj}\left( #1\right) }
\newcommand\projj[1]{\mathrm{p}_{#1} }
\newcommand\defeq{\colon =}
\def\prodd{\underline{\times} }

\newcommand\lan[1]{\mathrm{lan}_{#1} }

\def\initial{\mathsf{0} }
\def\terminal{\mathsf{1} }
\def\liva{\chi }
\def\livb{\gamma }
\def\livc{\zeta }
\def\livvv{\varphi }
\mathrmdef{id}

\usepackage{ifluatex}
\input diagxy        
\xyoption{curve}     
\xyoption{all}
\xyoption{2cell}

\begin{document}


\ifluatex
\directlua{adddednatlualoader = function ()
require = function (stem)
local fname = dednat6dir..stem..".lua"
package.loaded[stem] = package.loaded[stem] or dofile(fname) or fname
end
end}
\catcode`\^^J=10
\directlua{dofile "dednat6load.lua"}
\else
%
\def\diagxyto{\ifnextchar/{\toop}{\toop/>/}}
\def\to     {\rightarrow}
\def\defded#1#2{\expandafter\def\csname ded-#1\endcsname{#2}}
\def\ifdedundefined#1{\expandafter\ifx\csname ded-#1\endcsname\relax}
\def\ded#1{\ifdedundefined{#1}
    \errmessage{UNDEFINED DEDUCTION: #1}
  \else
    \csname ded-#1\endcsname
  \fi
}
\def\defdiag#1#2{\expandafter\def\csname diag-#1\endcsname{\bfig#2\efig}}
\def\defdiagprep#1#2#3{\expandafter\def\csname diag-#1\endcsname{{#2\bfig#3\efig}}}
\def\ifdiagundefined#1{\expandafter\ifx\csname diag-#1\endcsname\relax}
\def\diag#1{\ifdiagundefined{#1}
    \errmessage{UNDEFINED DIAGRAM: #1}
  \else
    \csname diag-#1\endcsname
  \fi
}
\newlength{\celllower}
\newlength{\lcelllower}
\def\cellfont{}
\def\lcellfont{}
\def\cell #1{\lower\celllower\hbox to 0pt{\hss\cellfont${#1}$\hss}}
\def\lcell#1{\lower\celllower\hbox to 0pt   {\lcellfont${#1}$\hss}}
\def\expr#1{\directlua{output(tostring(#1))}}
\def\eval#1{\directlua{#1}}
\def\pu{\directlua{pu()}}
%

\defdiag{pullbackcategoriesxx}{   
  \morphism(0,0)|a|/->/<750,0>[{P}`{W};{{j}}]
  \morphism(0,0)|b|/->/<0,-450>[{P}`{Z};{{h}}]
  \morphism(750,0)|r|/->/<0,-450>[{W}`{Y};{{f}}]
  \morphism(0,-450)|b|/->/<750,0>[{Z}`{Y};{{g}}]
}
\defdiag{pullbackcategoriesoverxtwocells}{   
  \morphism(0,0)|a|/->/<1050,0>[{d(t)}`{a\circ{j}(t)};{{(j,\livvv)}}]
  \morphism(0,0)|b|/->/<0,-450>[{d(t)}`{c\circ{j}(t)};{{\left(h,\livc\right)}}]
  \morphism(1050,0)|r|/->/<0,-450>[{a\circ{j}(t)}`{b\circ\left(f\circ{j}\right)(t)=b\circ\left(g\circ{h}\right)(t)};{{\left(f,\livb\right)}}]
  \morphism(0,-450)|b|/->/<1050,0>[{c\circ{j}(t)}`{b\circ\left(f\circ{j}\right)(t)=b\circ\left(g\circ{h}\right)(t)};{{\left(g,\liva\right)}}]
}
\defdiag{pullbackcategoriesoverx}{   
  \morphism(0,0)|a|/->/<750,0>[{(P,d)}`{(W,a)};{{(j,\livvv)}}]
  \morphism(0,0)|b|/->/<0,-450>[{(P,d)}`{(Z,c)};{{\left(h,\livc\right)}}]
  \morphism(750,0)|r|/->/<0,-450>[{(W,a)}`{(Y,b)};{{\left(f,\livb\right)}}]
  \morphism(0,-450)|b|/->/<750,0>[{(Z,c)}`{(Y,b)};{{\left(g,\liva\right)}}]
}
\defdiag{coequalizercategoriesoverxxx}{   
  \morphism(0,0)|b|/{@{->}@/_15pt/}/<750,0>[{(W,a)}`{(Y,b)};{{(g,\liva)}}]
  \morphism(750,0)|r|/->/<525,0>[{(Y,b)}`{(C,d)};{{(j,\livvv)}}]
  \morphism(0,0)|a|/{@{->}@/^15pt/}/<750,0>[{(W,a)}`{(Y,b)};{{(f,\livb)}}]
}
\defdiag{coequalizerofLan}{   
  \morphism(0,0)/{@{->}@/_15pt/}/<750,0>[{\lan{jf}a=\lan{jg}a}`{\lan{j}b};]
  \morphism(750,0)|a|/->/<300,0>[{\lan{j}b}`{d};{{{\livvv}^t}}]
  \morphism(0,0)/{@{->}@/^15pt/}/<750,0>[{\lan{jf}a=\lan{jg}a}`{\lan{j}b};]
}
\defdiag{paralellarrows}{   
  \morphism(0,0)|a|/->/<750,0>[{\lan{jg}a}`{\lan{jg}(b\circ{g})};{{\lan{jg}\livb}}]
  \morphism(750,0)/->/<600,0>[{\lan{jg}(b\circ{g})}`{\lan{j}b};]
  \morphism(0,-150)|b|/->/<750,0>[{\lan{jf}a}`{\lan{jf}(b\circ{f})};{{\lan{jf}\liva}}]
  \morphism(750,-150)/->/<600,0>[{\lan{jf}(b\circ{f})}`{\lan{j}b};]
}
\defdiag{coequalizercategories}{   
  \morphism(0,0)|b|/{@{->}@/_15pt/}/<750,0>[{W}`{Y};{{g}}]
  \morphism(750,0)|r|/->/<300,0>[{Y}`{C};{{j}}]
  \morphism(0,0)|a|/{@{->}@/^15pt/}/<750,0>[{W}`{Y};{{f}}]
}
\defdiag{pullbackcategories}{   
  \morphism(0,0)|a|/->/<750,0>[{P}`{W};{{j}}]
  \morphism(0,0)|b|/->/<0,-450>[{P}`{Z};{{h}}]
  \morphism(750,0)|r|/->/<0,-450>[{W}`{Y};{{f}}]
  \morphism(0,-450)|b|/->/<750,0>[{Z}`{Y};{{g}}]
}
\defdiag{pullbackobstruction}{   
  \morphism(0,0)/->/<750,0>[{V(y)}`{n};]
  \morphism(0,0)/->/<0,-450>[{V(y)}`{V(e)};]
  \morphism(750,0)/->/<0,-450>[{n}`{V(b)};]
  \morphism(0,-450)|b|/->/<750,0>[{V(e)}`{V(b)};{{V(q)}}]
}
\defdiag{pullbackobstructionL}{   
  \morphism(0,0)/->/<1350,0>[{L(Z)=(Z,\underline{\initial})}`{(W,d)};]
  \morphism(0,0)/->/<0,-450>[{L(Z)=(Z,\underline{\initial})}`{L(E)=(E,\underline{\initial})};]
  \morphism(1350,0)|r|/->/<0,-450>[{(W,d)}`{L(B)=(B,\underline{\initial})};{{(f,{\livb})}}]
  \morphism(0,-450)|b|/->/<1350,0>[{L(E)=(E,\underline{\initial})}`{L(B)=(B,\underline{\initial})};{{L(p)}}]
}
\defdiag{fibrationcat}{   
  \morphism(0,0)|a|/->/<0,-300>[{\sum{\Cat\left[{-,X}\right]}}`{{\Cat}//X};{\cong}]
  \morphism(0,0)|a|/->/<750,-150>[{\sum{\Cat\left[{-,X}\right]}}`{\Cat};{\mathsf{U}}]
  \morphism(0,-300)|b|/->/<750,150>[{{\Cat}//X}`{\Cat};{U}]
}
\defdiag{laxmorphismofcoalgebras1cell}{   
  \morphism(0,0)|a|/->/<300,0>[{W}`{Y};{f}]
}
\defdiag{twocelloflaxcommamorphism}{   
  \morphism(0,0)|a|/->/<600,0>[{W}`{Y};{f}]
  \morphism(0,0)|l|/->/<300,-420>[{W}`{X};{a}]
  \morphism(600,0)|r|/->/<-300,-420>[{Y}`{X};{b}]
  \morphism(128,-0)|a|/{@{=>}@<-20pt>}/<345,0>[{\phantom{O}}`{\phantom{O}};{\livb}]
}
\defdiag{compositionofmorphismslaxcommatwocategories}{   
  \morphism(0,0)|a|/->/<675,0>[{W}`{Y};{f}]
  \morphism(675,0)|a|/->/<675,0>[{Y}`{Z};{g}]
  \morphism(0,0)|l|/->/<675,-600>[{W}`{X};{a}]
  \morphism(1350,0)|r|/->/<-675,-600>[{Z}`{X};{c}]
  \morphism(675,0)|m|/->/<0,-600>[{Y}`{X};{b}]
  \morphism(240,-150)|a|/=>/<345,0>[{\phantom{O}}`{\phantom{O}};{\livb}]
  \morphism(765,-150)|a|/=>/<345,0>[{\phantom{O}}`{\phantom{O}};{\liva}]
}
\defdiag{leftsideequationtwocellforlaxcommacategor}{   
  \morphism(0,0)|a|/{@{->}@/^18pt/}/<750,0>[{W}`{Y};{f'}]
  \morphism(0,0)|b|/{@{->}@/_18pt/}/<750,0>[{W}`{Y};{f}]
  \morphism(0,0)|l|/->/<0,-675>[{W}`{X};{a}]
  \morphism(750,0)|r|/->/<-750,-675>[{Y}`{X};{b}]
  \morphism(375,150)|l|/<=/<0,-300>[{\phantom{O}}`{\phantom{O}};{\livc}]
  \morphism(-8,-338)|a|/=>/<390,0>[{\phantom{O}}`{\phantom{O}};{\livb}]
}
\defdiag{rightsideequationtwocellforlaxcommacategory}{   
  \morphism(0,0)|a|/->/<690,0>[{W}`{Y};{f'}]
  \morphism(0,0)|l|/->/<0,-645>[{W}`{X};{a}]
  \morphism(690,0)|r|/->/<-690,-645>[{Y}`{X};{b}]
  \morphism(22,-322)|a|/{@{=>}@<15pt>}/<375,0>[{\phantom{O}}`{\phantom{O}};{\livb'}]
}
\defdiag{equalizerdiagram}{   
  \morphism(0,0)|a|/{@{->}@<7pt>}/<1500,0>[{\displaystyle\prod_{w\in\objj{W}}{T(w,w)}}`{\displaystyle\prod_{(w,y)\in\objj{W\times{W}}}{\prod_{h\in{W(w,y)}}{T(w,y)}}};{t_0}]
  \morphism(0,0)|b|/{@{->}@<-7pt>}/<1500,0>[{\displaystyle\prod_{w\in\objj{W}}{T(w,w)}}`{\displaystyle\prod_{(w,y)\in\objj{W\times{W}}}{\prod_{h\in{W(w,y)}}{T(w,y)}}};{t_1}]
}

\def\pu{}
\fi	



\title{Lax comma categories: cartesian closedness, extensivity, topologicity,  and descent}
\author{Maria Manuel Clementino, Fernando Lucatelli Nunes and Rui Prezado}
\address{(1,2,3): University of Coimbra, CMUC, Department of Mathematics, 3000-143 Coimbra, Portugal\\(2): Department of Information and Computing Sciences, Utrecht University, The Netherlands}
\eaddress{mmc@mat.uc.pt, f.lucatellinunes@uu.nl, and ruiprezado@gmail.com}
\amsclass{18N10, 18N15, 18A05, 18A22, 18A40}

\keywords{lax comma categories, Grothendieck descent theory, Galois theory, $2$-dimensional category theory,
topological functor, effective descent morphism, cartesian closed category, exponentiability}

\maketitle

\begin{abstract}
We investigate the properties of lax comma categories over a base category $X$,
focusing on topologicity, extensivity, cartesian closedness, and descent. We establish that
the forgetful functor from $\Cat //X$ to $\Cat$ is topological if and only if
$X$ is large-complete. Moreover, we provide conditions for $\Cat //X $ to be complete,
cocomplete, extensive and cartesian closed. We analyze descent in $\Cat //X$ and
identify necessary conditions for effective descent morphisms. Our findings
contribute to the literature on lax comma categories and provide a foundation
for further research in $2$-dimensional Janelidze's Galois theory.
\end{abstract}

\setcounter{secnumdepth}{-1}

%
\pu

%
\pu

%
\pu

%
\pu

%
\pu	

%
\pu

%
\pu

%
\pu

%
\pu

%
\pu

%
\pu

%
\pu	
%
\pu

%
\pu	

%
\pu	
%
\pu

%
\pu

\section{Introduction}\label{section:introduction}

The ubiquitous notion of comma category has a natural $2$-dimensional lax
notion, called lax comma $2$-category (see, for instance,
\cite[Ch.~I,5]{zbMATH03447118}). The first two authors' motivation for investigating this notion stems from the fundamental role it plays in our approach to studying $2$-dimensional counterparts
of Janelidze's Galois theory~\cite{2020arXiv200203132C}, particularly regarding its interplay with lax orthogonal factorization systems~\cite{zbMATH06631665}.

While various remarkable insights exist in the literature (see, for instance, \cite{Fol70, Gui73, 2022arXiv221213541C, PT20}), we identified a gap in the systematic exploration of foundational properties crucial for advancing our research in the direction suggested in \cite{2020arXiv200203132C}.

The present paper builds upon our prior examination of lax comma ordered sets~\cite{2022arXiv221213541C}, extending the scope to encompass lax comma categories $\Cat // X $, where $\Cat$ denotes the category of small categories and $X$ represents a (possibly large) category. Our primary objective is to lay down the groundwork for our ongoing work in Galois theory and descent theory initiated in \cite{2020arXiv200203132C}. In pursuit of this aim, we concentrate on four fundamental aspects: cartesian closedness, extensivity, topologicity, and descent.

In Section \ref{section:basic-properties}, we start by giving conditions under
which the category $\Cat //X $ is cartesian closed, extensive and (co)complete,
showing that the properties of the base category $X$ play a crucial role in
determining the properties of $\Cat//X $. In this direction, we also establish
that, if $X$ is small, the forgetful functor $\Cat // X\to \Cat$ is topological
if and only if $X$ is complete (which is equivalent to say that $X$ is a
complete lattice). 

The study and characterization of effective descent morphisms have a rich and
intricate history. The work of Reiterman-Tholen~\cite{zbMATH00567702} in the
realm of topological spaces,  reformulated by
Clementino-Hofmann~\cite{zbMATH01766377}, exemplifies the depth and complexity
of this classification problem. Several notable contributions have significantly
advanced our understanding of effective descent morphisms in categories of
categorical structures. Notably, the study of effective descent morphisms in the
categories of internal
categories~\cite{PhDThesisIvanLeCreurerLouvainlaNeuve1999}, the investigation
into categories of enriched categories (see, for instance,
\cite{zbMATH06760394}, \cite[Theorem~9.11]{2016arXiv160604999L} and
\cite{2023arXiv230504042P}), and the study of effective descent morphisms in
categories of generalized multicategories (see, for instance,
\cite{zbMATH02153870, zbMATH07648507, 2023arXiv230908084P}) have each imparted
pivotal insights. In Section \ref{section:descent}, we initiate our
investigation into effective descent morphisms in $\Cat//X$. In this direction,
we show that the forgetful functor $\Cat // X\to \Cat$ preserves effective
descent morphisms. Despite the seemingly basic nature of $\Cat // X $, we unveil
that the full characterization of effective descent morphisms in this category
remains an open challenge. This work serves as the inception of our exploration
into this descent aspect of $\Cat // X $.

Our investigation on basic categorical properties of $\Cat // X$ presented
herein sheds light on the general properties of lax comma categories and lays
the foundation for our future work in generalized aspects of lax comma objects
in $2$-categories, particularly in the context of $2$-dimensional
Janelidze-Galois theory. In Section \ref{section:twodimensional-observations},
we give some further comments and point to future work.

This paper serves as a formal exposition of our findings,  contributing to the
basic literature on lax comma categories and lax comma $2$-categories in
general.

\paragraph{Acknowledgements}
We thank the anonymous referee for their insightful and prompt feedback, which included valuable remarks, suggestions, and additional references.

The authors acknowledge partial financial support by {\it Centro de Matemática
da Universidade de Coimbra} (CMUC), funded by the Portuguese Government through
FCT/MCTES, DOI 10.54499/UIDB/00324/2020.

This research was supported through the programme ``Oberwolfach Leibniz
Fellows'' by the Mathematisches Forschungsinstitut Oberwolfach in 2022.

\setcounter{secnumdepth}{5}
\section{Preliminaries}\label{section:preliminaries}
We denote by $\Cat $ the $2$-category of small categories, functors and natural transformations. In this section, we recall the explicit definition of lax comma categories $\Cat // X $. For additional general fundamental aspects, we recommend consulting \cite[I,5]{zbMATH03447118} and \cite{2020arXiv200203132C}.

\begin{defi}\label{definitionoflaxcommacategory}
Given a (possibly large) category $X$, we denote by
$\Cat // X  $ the category defined by the following.
\begin{itemize}
\renewcommand\labelitemi{--}	
\item The objects are pairs $ (W, a) $ in which $W$ is a small category
and $a: W\to X $ is a functor.
\item A morphism in $\Cat // X  $  between objects $(W,  a) $ and $ (Y, b) $    is a pair
$$\left( \diag{laxmorphismofcoalgebras1cell}, \diag{twocelloflaxcommamorphism}        \right) $$
in which
$f: W\to X $ is a functor and  $\livb $ is a natural transformation. 
		
If $(f, \livb ): (W,  a)\to (Y,  b)  $ and $ (g, \liva ): (Y, b)\to (Z, c) $ are morphisms  of $\Cat // X $, the composition is defined by
$\left( g\circ f,  \left(\liva \ast \id _f\right)\cdot \livb \right) $, that is to say,
the composition of the morphisms $g$ and $f$ with the pasting
$$\diag{compositionofmorphismslaxcommatwocategories} $$
of the natural transformations $\liva $ and $\livb $. Finally, with the definitions above, the identity on the object $(W,a)$ is the morphism $(\id _W, \id _a ) $.
\end{itemize}
Herein, the category $\Cat // X$ is called the \textit{lax comma category} over $X$. 
\end{defi}

A morphism $(f, \livb )$ in $\Cat // X $ is called \textit{strict} if $\livb $ is the identity. It should be noted that the comma category $\Cat / X $ is the subcategory of $\Cat // X $ consisting of the same objects but only the strict morphisms.

While the category $\Cat // X $ can be endowed with a natural two-dimensional structure, our current investigation primarily centers on its foundational one-categorical structure. We briefly touch upon the two-dimensional aspects in Section \ref{section:twodimensional-observations}. For the reader interested in two-dimensional features, we refer to \cite{2020arXiv200203132C, zbMATH07558575}. Further investigation of two-dimensional aspects is reserved for future research endeavors.

\subsection{Lax comma categories as total spaces}
We recall that lax comma categories are fibred over $\Cat $. This fibred structure over $\Cat$ is pivotal as it enables us to derive several essential properties and insights concerning limits, colimits, extensivity, topologicity, among other properties.
We refer to \cite{zbMATH03305157}, \cite[A1.1.7 and B1.3.1]{zbMATH01840601}, \cite{lucatellivakar2024} and \cite[Section~6]{lucatellivakar2023} for basic aspects on fibrations and Grothendieck constructions.

Given a (possibly large) category $X$, the forgetful functor defined in \eqref{eq:fibrationoverX} is a fibration. 
\begin{equation}\label{eq:fibrationoverX}
	U :  \Cat // X  \to \Cat , \qquad  (W, a)  \mapsto W , \qquad
	 (f, \livb )     \mapsto f .
\end{equation}
Denoting by  $\Cat\left[ -, X \right] :  \Cat ^\op \to \CAT $ 
the functor that takes every small category $W$ to the category $\Cat\left[ W, X \right]$ of functors and natural transformations, we have that
$$\diag{fibrationcat}$$
commutes, where we denote by $\displaystyle \sum \Cat\left[ -, X \right] $ the Grothendieck construction of the $2$-functor $\Cat\left[ -, X \right] $, and $\mathsf{U} $ is the associated fibration (forgetful functor).

\subsection{Adjoints}\label{subsection:adjoints}
Clearly, if $X$ has an initial object $\initial $, then \eqref{eq:fibrationoverX} has a left adjoint; namely
\begin{equation}\label{eq:fibrationoverX-leftadjoint}
	L :   \Cat \to  \Cat // X , \qquad  W \mapsto \underline{\initial} , \qquad
	(f:W\to Y)    \mapsto \left( f , \iota \right) 
\end{equation}
where $\underline{\initial} $ denotes the functor constantly equal to the initial
object, and $\iota : \underline{\initial}\rightarrow \underline{\initial}\circ f
$ is the only natural transformation.  Dually, \eqref{eq:fibrationoverX}
has a right adjoint if $X$ has a terminal object.


\section{Basic properties}\label{section:basic-properties}
In this section, we study cartesian closedness, extensivity and topologicity of lax comma categories.  As 
a foundational step in our study, we also provide an explicit construction of pullbacks, equalizers and coequalizers in lax comma categories. 

A fundamental underpinning for our investigation of limits and colimits in lax comma categories lies in the pivotal observation that the forgetful functor $U: \Cat // X \to \Cat$ is a fibration. For a more nuanced understanding of limits and colimits in the context of fibred categories, we recommend consulting \cite{zbMATH03305157}, \cite[Section~6]{lucatellivakar2023}, and \cite{lucatellivakar2024}.

\subsection{Cartesian closedness}
We start by studying the cartesian structure of the category $\Cat // X $. We
note that, whenever $X$ has products, for each small category $W$ and any
functor $g: W\to Y $, $\Cat\left[ W, X \right]$ has products, and the
change-of-base functors $\Cat\left[ g, X\right]$ preserve products. 

Since $\sum \Cat\left[ -, X \right]\cong \Cat // X $, by general results on
products of fibred categories, the observation above allows us to conclude
Propositions \ref{prop:terminal} and \ref{prop:products} (see, for instance,
\cite{zbMATH03305157, lucatellivakar2024} and
\cite[Section~6]{lucatellivakar2023} for further details). 
\begin{prop}\label{prop:terminal}
	If $X$ has a terminal object, then so does $\Cat // X $. 
	
	More precisely,  the terminal object of $\Cat // X$ is given by $\left(\terminal, \overline{\terminal } \right) $ where $\terminal $ is the terminal category, and  $\overline{\terminal } $ denotes the functor $\overline{\terminal } : \terminal \to X $ whose image is the terminal object in $X$.
\end{prop}
\begin{prop}\label{prop:products}
If $X$ has binary products then so does $\Cat// X $. 

More precisely, if $(W, a), (Y, b) $ are objects of $\Cat // X $, 
then the object $ (W\times Y, a\prodd b) $, where $a\prodd b (w,y) = a(w)\times b(y) $, is the product
$(W, a)\times (Y, b) $ in $\Cat // X $.

Furthermore, if $X$ has products, so does $\Cat // X$. More explicitly, if $\left( W_i, a_i \right) _{i\in L} $ is a family of objects in $\Cat // X$, 
$$ \prod _{i\in L} \left( W_i, a_i \right) \cong \left( \prod _{i\in L} W_i, \underline{a}  \right),   $$ 
where $\displaystyle\underline{a}  \left( x_i \right) _{i\in L} = \prod _{i\in L} a_i (x_i) $.
\end{prop}

Being a natural generalization of the setting of \cite{2022arXiv221213541C}, the study of exponentials  fits the setting of \cite{lucatellivakar2024}.
However, we stick to an explicit presentation of the exponentials. More precisely, to streamline the computation of exponentials in $\Cat // X$, we employ the concept of ends, as outlined in \cite[3.10]{zbMATH03751225} and recalled below. 

Assuming that $X$ is cartesian closed and $x,y$ are objects of $X$, we denote by $$\projj{x}: x\times y \to x $$ the projection, and by $x\Rightarrow y $ the exponential; that is to say, 
$$ x\Rightarrow - : X  \to X $$
is the functor right adjoint to $ x\times - : X\to X $.

 It is worth noting that a \textit{complete category inherently possesses all ($\Set$-enriched/ordinary) ends}. Specifically, for any small category $W$ and any functor $$T: W^\text{op} \times W \to X,$$ the end $\displaystyle \int_{W} T   = \int_{w \in W} T (w,w) $ is given by the equalizer of the following diagram:
$$ \diag{equalizerdiagram}$$
where, for the component corresponding to $(w,y,h)\in \objj{W\times W}\times W(w,y)$, we have that $t_0 $  and $t_1 $  are respectively induced by
$ T(w,h)\cdot \projj{T(w,w)}$ and $ T(h,y)\cdot \projj{T(y,y)}$.

\begin{theo}
	Let $X$ be a cartesian closed category. We assume that $(W,a), (Y,b)$ are objects in $\Cat // X $
	such that \eqref{eq:end-btothea} exists in $X$ for any functor $h: W\to Y $. 
	In this setting, the exponential $(W, a)\Rightarrow (Y, b) $ in $\Cat//X$ exists and is given by $\left( \Cat \left[ W, Y\right], b ^a    \right)$.
\begin{equation}\label{eq:end-btothea}
		b^a (h) \defeq\int _{w\in W} \left( a(w)\Rightarrow b\cdot h (w)\right)
\end{equation}
Consequently, if $X$ is a complete cartesian closed category, so is $\Cat // X $.
\end{theo} 
\begin{proof}
	Indeed, we have that:
\begin{eqnarray*}
	\Cat // X \left(  (W,a)\times (Z,c), (Y, b) \right) & \cong & \Cat // X \left(  (W\times Z, a\prodd c), (Y, b) \right) \\ 
	& \cong & \coprod _{f\in \Cat \left( W\times Z, Y\right) } \Cat \left[ W\times Z, X\right]  \left( a\prodd c , b\cdot f \right) \\
	& \cong & \coprod _{f\in \Cat \left( W\times Z, Y\right) } \int _{(w,z)\in W\times Z} X \left( a(w)\times c(z), b(f(w,z)) \right)  \\
	& \cong & \coprod _{f\in \Cat \left( W\times Z, Y\right) } \int _{z\in Z}\int _{w\in W} X \left( c(z), a(w)\Rightarrow b(f(w,z)) \right) \\
	& \cong & \coprod _{\dot{f}\in \Cat \left( Z, \Cat\left[ W, Y\right]\right) }  \int _{z\in Z} X\left( c(z), \int _{w\in W}\left( a(w)\Rightarrow b(\dot{f}(z)(w))\right) \right) \\
		& \cong & \coprod _{\dot{f}\in \Cat \left( Z, \Cat\left[ W, Y\right] \right) }  \int _{z\in Z} X\left( c(z), b^a \left( \dot{f} (z) \right) \right) \\
& \cong & \coprod _{\dot{f}\in \Cat \left( Z, \Cat\left[ W, Y\right]\right) }  \Cat \left[ Z, X\right]  \left( c(z), b^a \left( \dot{f} (z) \right) \right) \\
& \cong & \Cat // X \left( \left( Z,c \right), \left( \Cat \left[ W, Y\right], b ^a  \right)  \right) 
\end{eqnarray*}	
\end{proof}

\subsection{Pullbacks}
We assume that \eqref{eq:pullbackcategories}  is a pullback in $\Cat $.
\begin{equation}\label{eq:pullbackcategories} 
\diag{pullbackcategories}
\end{equation}
Since the change-of-base functors of the indexed category $\Cat \left[ -, X\right]$ preserve 
limits (and, in particular, pullbacks), we have that:
\begin{theo} 
If $X$ is a category with pullbacks, then $\Cat // X $ has pullbacks. Explicitly, 
\eqref{eq:pullbackcategoriesoverx}  is a pullback in $\Cat // X $ provided that
\eqref{eq:pullbackcategories}  is a pullback in $\Cat $ and \eqref{eq:pullbackcategoriestwocells}  
is a pullback for any object $t$ in $P$.
\\
\noindent\begin{minipage}{.6\linewidth}
\begin{equation}\label{eq:pullbackcategoriestwocells} 
	\diag{pullbackcategoriesoverxtwocells}
\end{equation}
\end{minipage}%
\begin{minipage}{.4\linewidth}
\begin{equation}\label{eq:pullbackcategoriesoverx} 
	\diag{pullbackcategoriesoverx}
\end{equation}
\end{minipage}
\end{theo}

\subsection{Coproducts and Extensivity}
Let $X$ be a category with an initial object $\initial $. For any small category
$W$, the fibre $\Cat\left[W, X \right] $ has an initial object
$\overline{\initial} : W\to X $ given by the functor constantly equal to
$\initial $. Moreover,  clearly, the indexed category 
\begin{equation}\label{eq:indexed-catoverx} 
\Cat \left[-, X \right] : \Cat ^\op \to \CAT 
\end{equation} 
is (infinitary) extensive in the sense of \cite[Section~6.6]{lucatellivakar2023}; namely, we mean that $\Cat \left[-, X \right]$ preserves products of $\Cat ^\op$, and $\Cat$ has coproducts.
Therefore, by \cite[Corollary~35]{lucatellivakar2023}, we have:

\begin{prop}\label{prop:coproducts-total-category}
	If $X$ is a category with an initial object $\initial $, then $\Cat // X $ has coproducts.
	
	Explicitly, the category $\Cat // X $ has an initial object
	defined by $\left( \initial , \overline{\initial} : \emptyset \to X  \right) $.
	
	Furthermore, if $(W_i, a_i)_{i\in L}$ is a family of objects in $\Cat // X $, then
	$$\coprod _{i\in L} (W_i, a_i) \cong \left( \coprod _{i\in L} W_i , \left[ a_i\right] _{i\in L}  \right)$$
	where  $\displaystyle\left[ a_i\right] _{i\in L} : \coprod _{i\in L} W_i \to X $ is the induced functor.	
\end{prop}
Since $\Cat $ is (infinitary) extensive, we can also conclude that:
\begin{theo}
	If $X$ is a category with an initial object $\initial $, then $\Cat // X $ is (infinitary) extensive.
\end{theo}
\begin{proof}
	The result follows from an infinitary version of the argument given in \cite[Theorem~41]{lucatellivakar2023}.
\end{proof}	

\subsection{Coequalizers}
To study the coequalizers of $\Cat // X $, we start by observing that \eqref{eq:fibrationoverX} is a bifibration provided that $X$ is cocomplete. More conveniently put in our context, this means that:
\begin{equation}
	\Cat \left[ f , X \right]: \Cat \left[ Y , X \right]\to \Cat \left[ W , X \right]
\end{equation} 
has a left adjoint for any functor $f: W\to Y $. More precisely, assuming that $X$ is cocomplete, $	\Cat \left[ f , X \right]$ has a left adjoint given by the pointwise left Kan extension 
\begin{equation}
	\lan{f} : \Cat \left[ W , X \right]\to \Cat \left[ Y , X \right] .
\end{equation} 
We refer the reader to \cite{zbMATH03362059, zbMATH03751225, 2016arXiv160604999L}  for pointwise Kan extensions.
With this observation in mind, we get:
\begin{theo}
		If $X$ is cocomplete, so is  $\Cat // X $. 
	
	Explicitly,	we have that \eqref{eq:coequalizer-catoveroverx} is a coequalizer in $\Cat // X $ 
	provided that:
	\begin{itemize}[--]
		\item \eqref{eq:coequalizer-cat}  is the coequalizer in $\Cat $;
		\item \eqref{eq:coequalizerofLan} is the coequalizer of the parallel morphisms \eqref{eq:coequalizerofLanarrows} in $\Cat \left[ C , X \right] $, where  each unlabeled arrow in \eqref{eq:coequalizerofLanarrows} is induced by the appropriate counit,    and $\livvv ^t$ is the mate of $\livvv $ w.r.t the adjunction $\lan{j}\dashv \Cat\left[ j , X \right] $.
	\end{itemize}	
\noindent\begin{minipage}{.5\linewidth}
	\begin{equation}\label{eq:coequalizerofLan} 
		\diag{coequalizerofLan}
	\end{equation}
\end{minipage}%
\begin{minipage}{.5\linewidth}
	\begin{equation} \label{eq:coequalizerofLanarrows}  
		\diag{paralellarrows}
	\end{equation} 
\end{minipage}		 
\\
\noindent\begin{minipage}{.4\linewidth}
\begin{equation}\label{eq:coequalizer-cat} 
	\diag{coequalizercategories}
\end{equation}
\end{minipage}%
\begin{minipage}{.6\linewidth}
\begin{equation} \label{eq:coequalizer-catoveroverx}  
	\diag{coequalizercategoriesoverxxx}
\end{equation} 
\end{minipage}		
\end{theo} 	
\begin{proof}
	The coproducts are given by Proposition \ref{prop:coproducts-total-category}. 
	
	As for the coequalizers, the direct verification is straightforward, and the result follows from general results on the total categories of bifibrations (see \cite{zbMATH03305157, lucatellivakar2024}). 
\end{proof}

\subsection{Topologicity}
In order to examine the topologicity of the functor $U: \Cat // X\to \Cat $, we fully rely on the 
characterization of \cite[Theorem~5.9.1]{zbMATH06334036}. In other words, \textit{$U: \Cat // X\to \Cat $ is topological if, and only if,
$U$ is a bifibration whose fibres are large-complete}.

\begin{theo}
	$U: \Cat // X\to \Cat $ is topological if, and only if, $X$ is large-complete.
\end{theo}
\begin{proof}
	If $X$ is large-complete, then so is every fibre \( \Cat(A,X) \) of $U$.
	Reciprocally, if \( U \) is topological, then the fiber \( \Cat(1,X) \cong X\)
	is large-complete.
\end{proof} 	

\begin{coro}
	Let $X$ be a small category. $U: \Cat // X\to \Cat $ is topological if, and
	only if, $X$ is a complete lattice.
\end{coro}

\section{Effective descent morphisms}\label{section:descent}
The study of effective descent morphisms, exhibiting a pivotal role in Grothendieck descent theory, has far-reaching implications across various subjects (see, for instance, \cite{zbMATH01577082, zbMATH01038604, zbMATH02069040}). Beyond their instrumental role, effective descent morphisms stand out as a fascinating subject on their own right, representing a distinctive subclass of stable regular epimorphisms. By definition, they are morphisms that show how bundles/morphisms over their codomain can be characterized as bundles over their domain endowed with an algebraic structure. We refer the reader to \cite{zbMATH02068090, zbMATH00626366, 2016arXiv160604999L, zbMATH07361310} for further aspects of effective descent morphisms. 

We recall that, explicitly, a morphism $q : u\to v $ in a category $X $ with pullbacks along $q$ is of effective descent if 
the change-of-base functor
\begin{equation} 
	q ^\ast : X / v \to X/ u
\end{equation} 	 
is monadic. Moreover, we are usually also interested in other cases, even if $q ^\ast$ is not monadic. More precisely,  if the Eilenberg-Moore comparison functor is fully faithful (respectively, faithful), $q$ is said to be of descent (respectively, almost descent). 
In this section, we embark on the study of effective descent morphisms in the lax comma category $\Cat //X$.

\subsection{Fundamentals}
Before delving into our study, we shortly recall the basic technique to study (effective) descent morphisms in a category. We refer the reader to \cite[Section~1]{2016arXiv160604999L} for a further overview.

If $q$ is a morphism of a category $X$ with pullbacks, one can study if $q$ is effective descent by investigating the monadicity of 
$q^\ast $ via Beck's monadicity theorem (see, for instance, \cite[Chapter~VI]{zbMATH03367095}, \cite[Corollary~1.2]{zbMATH06585659}, \cite{zbMATH07629358}, and \cite{zbMATH07361310}). This proves to be quite fruitful in some special cases: for instance, one can conclude that, in locally cartesian closed categories, effective descent morphisms are the same as the stable regular epimorphisms. 

Beyond Beck's monadicity theorem, most general results we have are about reflection of properties by functors.
More precisely, pullback-preserving fully faithful functors do not generally reflect all effective descent morphisms, but provide us with the following classical result (see, for instance, \cite[2.6]{zbMATH00626366} and \cite[Theorem~1.3]{2016arXiv160604999L}).


\begin{theo}\label{theo:effective-descent-obstruction}
	Let $V: X\to N $ be a fully faithful pullback-preserving functor between categories with pullbacks.
	We assume that $V(q : e\to b) $ is an effective descent morphism. The morphism $q$ is of effective descent in $X$ if and only if it satisfies the following property: whenever \eqref{diag:pullback-V-p} is a pullback in $N$, there is an object $x$ in $X$ such that $V(x)\cong n $.
\begin{equation}\label{diag:pullback-V-p}
	\diag{pullbackobstruction}
\end{equation}	    
\end{theo}

The result above is one of the reasons why the main technique to study (effective) descent morphisms in a category $X$ is about fully embedding $X$ into a category $N$ whose descent behaviour we know about. We refer the reader to \cite[Section~1]{2016arXiv160604999L} for further observations about functors reflecting descent properties of morphisms.

\subsection{Preservation of effective descent morphisms}
Although the main classical results in the study of effective descent morphisms,
such as Theorem \ref{theo:effective-descent-obstruction}, focus on understanding
whether they are reflected by fully faithful, pullback preserving functors, our
main result, Theorem \ref{theo:topological-functor-U}, is about
\textit{preservation} of effective descent morphisms. 

Throughout this subsection, we assume that \textit{$X$ has a strict initial object $\initial $}.
We start by considering the left adjoint  
\begin{equation}\label{eq:L-left-adjoint} 
L: \Cat \to \Cat // X
\end{equation}  
to the forgetful functor $U: \Cat //X \to \Cat$,
defined in \ref{subsection:adjoints}. That is to say: 
\begin{lem}
  $L$ is fully faithful and preserves limits of non-empty diagrams. In particular, it preserves  
  pullbacks and non-empty products. 
  
  $L$ preserves the terminal object if, and only if, $X$ is equivalent to the terminal category.
\end{lem}
\begin{proof}
    We have \( UL(C) = C \) for any small category \( C \), hence \( L \) is
    fully faithful.

    Since the initial object in $X$ is strict by hypothesis, we get that the limit of any non-empty diagram
    involving the initial object is isomorphic to the initial object. Therefore, it follows by the constructions of pullbacks and non-empty
    products in $\Cat // X$ (see Theorem \ref{eq:pullbackcategories}  and
    Proposition \ref{prop:products}) that $L$ preserves pullbacks and non-empty
    products.

    Moreover, since $L \left( \terminal \right) = \left( \terminal, \initial
    \right)  $, we conclude by Proposition \ref{prop:terminal} that $L \left(
    \terminal \right) $ is the terminal object if and only if $\initial \cong \terminal
    $ in $X$.   

    Since $X$ has a strict initial object by hypothesis, then having zero object
    is equivalent to having $ X \simeq 1 $.
\end{proof}

\begin{theo}\label{theo:reflection-L-absolute} 
	The functor $L$ reflects effective descent morphisms.
\end{theo} 
\begin{proof}
	Given any functor $p : E\to B $ in $\Cat $, we have that, if \eqref{diag:pullback-L-p}
	is a pullback in $\Cat // X $, then $d = \underline{\initial }$ since $\initial $ is strict 
	and $\livb : d \Rightarrow \underline{\initial } $.
\begin{equation}\label{diag:pullback-L-p}
	\diag{pullbackobstructionL}
\end{equation}	 		
This proves that, by Theorem \ref{theo:effective-descent-obstruction}, $p$ is of effective descent provided that $L(p)$ is of effective descent.
\end{proof}		

Our first result on the preservation of effective descent morphisms follows from
the well-known result that effective descent morphisms are stable under
pullback, that is, if $p$ is of effective descent, then so is $q^\ast (p)$ in
any category \(X\).

\begin{lem}\label{lem:preservation1} 
	$LU(f,\livb)$ is of effective descent if $(f,\livb)$ is of effective descent.
\end{lem} 	 
\begin{proof} 
	Consider $(\id _{Y}, \iota _b ) : \left( Y, \underline{\initial }\right)\to
	\left( Y, b   \right) $ where $\iota _b $ is the unique $2$-cell $\iota _b :
	\underline{\initial} \Rightarrow b$.
		
	We have that $LU(f, \livb) = \left( \id _{Y}, \iota _b \right) ^\ast  (f,
	\livb) $ and, hence, whenever $(f, \livb)$ is of effective descent so is the
	pullback $ LU(f, \livb) $.
\end{proof}

As a consequence, we get:

\begin{theo}\label{theo:topological-functor-U}
	The functor $ U : \Cat // X \to \Cat $ preserves effective descent morphisms provided that 
	$X$ has pullbacks and a strict initial object.
\end{theo}
\begin{proof}
	If $\left(f, \livb\right) $ is of effective descent in $\Cat // X $, then so is $LU(f, \livb)$ by Lemma \ref{lem:preservation1}. Finally, we can conclude that $U(f, \livb) = f$ is an effective descent morphism by
	Theorem \ref{theo:reflection-L-absolute}.
\end{proof}

\section{Further comments}\label{section:twodimensional-observations}
We have pointed out that it is natural to consider $\Cat // X $ as a
$2$-category. We refer the reader to \cite{2020arXiv200203132C, zbMATH07558575}
and future work for more details. We recall the $2$-dimensional structure of
$\Cat // X $ below. 

\begin{defi}\label{definitionoflaxcommacategorytwodimensional}
Given a small category $X$, we consider  $\Cat // X  $ as a $2$-category with
the $2$-cells given by the following.
\begin{itemize}[label=--]
    \item 
        A $2$-cell between morphisms $(f, \livb  )  $ and $(f', \livb ' )  $ is
        given by a $2$-cell $\livc : f\Rightarrow f' $ such that the equation
        \begin{equation*}
        	\diag{leftsideequationtwocellforlaxcommacategor}\quad =\quad
        	\diag{rightsideequationtwocellforlaxcommacategory}
        \end{equation*}
        holds.
\end{itemize}
The $2$-category $\Cat // X$ is called the \textit{lax comma $2$-category} of
$\Cat $ over $X$, while we call the underlying category the \textit{lax comma
category} $\Cat // X $.
\end{defi}

\subsection{Monadicity and lax monadicity}
It is well known that the comma category $\Cat / X$ is isomorphic to the
category of coalgebras and (strict) morphisms of the ($2$-)comonad $T_X$ whose
underlying ($2$-)endofunctor is given by $W \mapsto W\times X $. In particular,
this means that $\Cat / X$ is cocomplete and the forgetful functor $\Cat /X \to
\Cat $ is $2$-comonadic. 

$\Cat //X$ fits into this picture as the $2$-category of strict $T_X$-coalgebras
and lax $T_X$-morphisms (see, for instance, \cite{zbMATH06970806} for basic
definitions). This observation shows that it would be interesting to understand
better the general interest in studying descent in categories of (strict)
algebras and lax morphisms between them.

\subsection{Free fibrations}
With the definition above, we have that $\Cat // X $ can be naturally embedded
in $\CAT\left[ X^\op ,  \Cat\right] $, as $\CAT// X $ is equivalent to the
$2$-category of free fibrations over $X$. This embedding preserves pullbacks
and, hence, it can be used to further get results on effective descent morphisms
in $\Cat // X $. However, a strategy to get a characterization would probably
rely on a generalization of the techniques presented in
\cite{2023arXiv231214315C}.

\subsection{Preservation of effective descent morphisms}
For basic definitions involved in the comment below, we refer the reader to
\cite[Section~1]{2016arXiv160604999L}.
As pointed out in Section \ref{section:descent}, on one hand, in the study of
descent, one usually relies on results about the reflection of properties. For
instance, besides Theorem \ref{theo:effective-descent-obstruction}, we recall
that:
\begin{itemize}[--]
    \item 
        $G$ reflects almost descent morphisms if it is a pullback-preserving
        faithful functor;
    \item 
        $G$ reflects descent morphisms if it is a pullback-preserving fully
        faithful functor.
\end{itemize}
On the other hand, there aren't many tools in the literature about the
preservation of effective descent morphisms. In this sense, Theorem
\ref{theo:topological-functor-U} gives rise to the question of whether one can
find a general setting or framework for which the preservation of effective
descent morphisms holds. We leave this investigation to future work.

\subsection{Effective descent morphisms w.r.t. other indexed categories}
We left the problem of fully characterizing effective descent morphisms (w.r.t.
the basic fibration) in  $\Cat // X$ open. 

We present, herein, another natural problem on effective descent morphisms
arising from our considerations; namely, for each $2$-category $\mathbb{A}$ and
object $W$ of $\mathbb{A}$, we have the indexed category
$$\mathbb{A}\left( - , W \right) : \mathbb{A}_0 ^\op  \to \CAT ,  $$
where $\mathbb{A}_0 $ is the underlying category, and $\mathbb{A}\left( X , W
\right)$ is the category of morphisms $X\to W $ in~$\mathbb{A}$. 

In particular, considering $\Cat // X$ as a  $2$-dimensional category, we have
the indexed category 
$$ \Cat // X \left( -,  \left( W, a\right)  \right) : 
    \left( \Cat // X\right)^\op \to  \CAT  $$
for each $\left( W, a\right) $ in $\Cat // X$. More relevantly, we can
extrapolate and consider the indexed category 

$$ \mathfrak{F} = \Cat // X \left( -,  \left( X, \id _X \right)  \right) :
\left( \Cat // X\right) ^\op \to  \CAT  .$$

Akin \cite{zbMATH02153876, zbMATH07733558}, it is natural to consider the
problem of characterizing effective $\mathfrak{F}$-descent morphisms. Following
the insights of \cite{zbMATH07733558}, we understand that such a study will rely
on understanding lax epimorphisms in $\Cat // X$ (see, for instance,
\cite{zbMATH01668934, zbMATH07559531, zbMATH07733558}). We leave this study with
further $2$-dimensional considerations for future work.

\bibliographystyle{plain-abb}


\pu

\end{document}

